%% file: main.tex
\documentclass[12pt,a4paper,leqno]{amsart}

\usepackage{amssymb,hyperref}
\usepackage{array,float}
\usepackage{amsmath,amsthm,amsbsy,amscd,mathrsfs}
\usepackage{graphicx}
\usepackage{mathtools}
\usepackage{cite}

\makeatletter
\newcommand\suchthat{%
 \@ifstar
  {\mathrel{}\middle|\mathrel{}}
  {\mid}%
}
\makeatother

\newtheorem{thm}{Theorem}[section]
  \newtheorem{prop}[thm]{Proposition}
  \newtheorem{cor}[thm]{Corollary}
  \newtheorem{lem}[thm]{Lemma}
   \theoremstyle{remark}
\newtheorem{rem}[thm]{Remark}
\numberwithin{equation}{section}  
\newtheorem{defi}[thm]{Definition}  

\theoremstyle{definition}
 \newtheorem{ex}{Example}[section]

 \DeclareMathOperator{\Pic}{Pic}
  \DeclareMathOperator{\Div}{Div}
  \DeclareMathOperator{\Picw}{\widetilde{Pic}}
    \DeclareMathOperator{\Divw}{\widetilde{Div}}
  \DeclareMathOperator{\To}{T^0}
  
  \DeclareMathOperator{\Tow}{\widetilde{T}^0}
    
  \DeclareMathOperator{\Tr}{Tr} 
  \DeclareMathOperator{\co}{covol}  
    \DeclareMathOperator{\vo}{vol} 
 
\title[ A generalization of reduced Arakelov divisors of a number field]{ A GENERALIZATION OF REDUCED ARAKELOV DIVISORS OF A NUMBER FIELD}

\author[Ha Thanh Nguyen Tran]{Ha Thanh Nguyen Tran} 
\address{Department of Mathematics and Systems Analysis,
Aalto University School of Science\\
Otakaari 1, 02150 Espoo, Finland.}
\email{hatran1104@gmail.com}


\begin{document}

\begin{abstract}
 Let $C \geq 1$. Inspired by the LLL-algorithm, we define strongly $C$-reduced divisors of a number field $F$ which are generalized from the concept of reduced Arakelov divisors. Moreover, we prove that strongly $C$-reduced Arakelov divisors still retain outstanding  properties of  the reduced ones: they form a finite, regularly distributed set in the Arakelov class group and the oriented Arakelov class group of $F$.

\end{abstract}

\maketitle


\section{Introduction}
\input{intro}

\section{Preliminaries}
\input{basic1}
\input{basic2a}

\input{basic2b}

\section{Strongly $C$-reduced Arakelov divisors}
\input{basic3}

\section{Properties of strongly $C$-reduced Arakelov divisors}\label{sec:cred}
\input{pro1}
\input{thm1}

\input{theo12_2}
\input{thm2}

\input{cardired}

\section*{Acknowledgement}
I would like to thank  Hendrik W. Lenstra for helping me to prove Theorem \ref{theo:dis2} and  Ren\'{e} Schoof for very useful comments. The author also would like to thank the reviewers for their comments that helped improve the manuscript. I would also like to show my gratitude to the Mathematics department at the University of Leiden for its hospitality during the fall 2013.

This research was supported by the Universit\`{a} di Roma ``Tor Vergata". The author is financially supported by the Academy of Finland grants $\#$276031, $\#$282938 and $\#$283262. The support from
 the European Science Foundation under the COST Action IC1104 is also gratefully acknowledged.
%



\end{document}

%% file: intro.tex
Let $F$ be a number field of degree $n$. The Arakelov class group $\Pic^0_F$ of $F$ is an analogue of the Picard group  of an algebraic curve. Computing this group is of interest since one can read off the class number and the unit group of $F$ from it (see  \cite{ref:9,ref:5,ref:8,ref:36}). A good tool to compute $\Pic^0_F$ is the \textit{reduced Arakelov divisors} the theory of which is also called \textit{infrastructure} of $F$ (see \cite{ref:10,ref:7,ref:18}). The set of these reduced Arakelov divisors has a group-like structure that enables to obtain a reduced divisor by performing reduction on the sum of two reduced divisors. This phenomenon was first discovered through computing regulators of real quadratic fields (see \cite{ref:7}).

Reduced Arakelov divisors form a finite and regularly distributed set in $\Pic^0_F$ (see Section 7 in \cite{ref:4} and Chapter 1 in \cite{ref:36}). Therefore, for any class of divisors $D$ of $\Pic^0_F$, we can always find a reduced divisor $D'$ that is close to $D$ and use it to compute at $D$. There is an effective method to find such a reduced divisor $D'$ {\cite[Algorithm 10.3]{ref:4}}. However, this algorithm requires finding a shortest vector of the lattice associated to $D$ which is a very time-consuming process. On the other hand, computing a reasonable short vector, such as using the LLL-algorithm, is much faster and easier than finding a shortest vector. This leads to modifications and generalizations of the concept of reduced divisors to $C$-reduced divisors.

The first the generalization is proposed by Schoof (see {\cite[Chapter 2]{ref:36}} and \cite{ref:33}) which is called \textit{$C$-reduced Arakelov divisor}. However, for general number fields, it is unknown how to efficiently test whether a given divisor is $C$-reduced. This paper focuses on the second generalization -- \textit{strongly $C$-reduced Arakelov divisors}. It is inspired by the properties of LLL-reduced bases of the lattices associated to Arakelov divisors. 

With this definition, testing whether a given divisor is strongly $C$-reduced can be done in time polynomial in  $\log(|\Delta_F|)$,  where $\Delta_F$ is the discriminant of $F$ (see Section 3.3 in \cite{ref:36} for more details). Especially, the LLL-algorithm yields strongly $C$-reduced Arakelov divisors with $C= \sqrt{n} \hspace*{0.1cm} 2^{(n-1)/2}$ (see Example \ref{excred}). 

Strongly $C$-reduced divisors are $C$-reduced (see \cite{ref:33}). In addition, a reduced divisor in the usual sense is strongly $C$-reduced with $C = \sqrt{n}$ and a strongly $C$-reduced divisor is reduced in the usual sense  with $C = 1$. Furthermore, strongly $C$-reduced divisors still admit the same remarkable properties as reduced divisors in the sense that they form a finite and regularly distributed set in $\Pic^0_F$. These are the most important results of this paper presented in Theorem \ref{theo:dis11}, \ref{theo:dis12} and \ref{theo:dis2}.

In Section 2, we briefly recall some definitions and basic properties of Arakelov divisors. Section 3 introduces strongly $C$-reduced divisors and their properties are provided in Section 4.

%% file: basic1.tex
This section is devoted to introducing Arakelov divisors and reduced Arakelov divisors of a number field $F$. The Arakelov class group and the oriented Arakelov class group as well the metrics on these groups are briefly recalled. All details can be found in \cite{ref:4} and \cite{ref:36}.

Let $F$ be a number field of degree $n$ and $r_1, r_2$ the number of real and complex infinite primes (or infinite places) of $F$. 
\subsection{Arakelov divisors and reduced Arakelov divisors} \label{sec21}\qquad\\
Let $F_{\mathbb{R}}: = F \otimes_\mathbb{Q}\mathbb{R} \simeq \prod_{\sigma \text{ real}}\mathbb{R} \times \prod_{\sigma \text{ complex}}\mathbb{C}$ with $\sigma$ running over the infinite primes of $F$. Then $F_{\mathbb{R}}$ is an \'{e}tale $\mathbb{R}$-algebra with a canonical Euclidean structure given by the scalar product
$$\langle u, v \rangle := \Tr(u \overline{v}) \text{ for any } u = (u_{\sigma}), v = (v_{\sigma}) \in F_{\mathbb{R}}.$$

Here and in the rest of the paper we often call fractional ideals
simply `ideals'. To emphasize that an ideal is integral, we call it an integral ideal.
  \begin{defi}\label{defA}
 An \textit{Arakelov divisor} is a formal finite sum
  $D=\sum\limits_{\mathfrak{p}}n_{\mathfrak{p}}\mathfrak{p} +\sum_{\sigma}x_{\sigma}\sigma$
    where $\mathfrak{p}$ runs over the nonzero prime ideals of $ O_F$ and $\sigma$ runs over the infinite primes of $F$, here the coefficients $n_{\mathfrak{p}}$ are in $ \mathbb{Z}$ but  the $x_{\sigma}$ can be any number in $ \mathbb{R}$.
 \end{defi}
 The set of all Arakelov divisors of $F$ is an additive group denoted by $\Div_F$.
   The degree of an infinite prime $\sigma$ is equal to $1$ or $2$ depending on whether $\sigma$ is real or complex. The \textit{degree} of $D$ is $deg(D): =\sum_{\mathfrak{p}}n_{\mathfrak{p}}\log{N(\mathfrak{p})} +\sum_{\sigma } deg(\sigma)x_{\sigma} $.

 Let $I $ be a fractional ideal of $F$. 
 Then each element $g$ of $I$ is mapped to the vector $(\sigma(g))_{\sigma}$ in  $F_{\mathbb{R}}$ by the infinite primes $\sigma$. For any vector $u =(u_{\sigma})_{\sigma}$ in  $F_{\mathbb{R}}$ and any $g \in I$, the vector $u g = (u_{\sigma} \sigma(g)) $ is in $F_{\mathbb{R}}$, so then 
  $$\|u g\|^2 = \sum_{\sigma \text{ real } }  u_{\sigma}^2 |\sigma(g)|^2 + 2\sum_{\sigma \text{ complex }} |u_{\sigma}|^2|\sigma(g)|^2.$$

To every Arakelov divisor $D$ as in Definition \ref{defA},  there corresponds a \textit{Hermitian line bundle} $ (I, u)$ where $I = \prod_{\mathfrak{p}}\mathfrak{p}^{-n_{\mathfrak{p}}}$ a fractional ideal of $F$ and 
$u = (e^{-x_{\sigma}})_{\sigma} $ a vector in $  \prod_{\sigma}\mathbb{R}_{+}^*$. This correspondence is bijective and we often identify
the two notions.

Additionally, we associate to $D$ the \textit{lattice} $L= u I = \{ u f: f \in I \}\subset F_{\mathbb{R}} $ with the inherited  metric from $F_{\mathbb{R}}$ (see more about ideal lattices in \cite{ref:0}). The covolume of this lattice $L= u I$ is $\co(L) = \sqrt{|\Delta|} e^{-deg(D)}$.
We define 
$\| f\|_D^2:= \| uf\|^2 \text{ for all } f \in I.$

For each element $f \in F^*$, the \textit{principal} Arakelov divisor $(f)$ is the Arakelov divisor of Hermitian line bundle $ (f^{-1}O_F, |f|)$ where 
$ f^{-1} O_F$ is the principal ideal generated by $f^{-1}$ and 
$|f| = (|\sigma(f)|)_{\sigma}$. The product formula implies that $deg((f)) =0$ for all $f \in F^*$.
\begin{defi}
	Let $I$ be a fractional ideal of $F$. An element $f$ in $I$ is called \textit{minimal} if for all $g \in I$ if $|\sigma(g)| < |\sigma(f)|$ for all $\sigma$ then $g=0$.
\end{defi}

\begin{defi}
	Let $I$ be a fractional ideal of $F$. We define $d(I)$ to be the divisor $(I,u)$ with $u = (u_{\sigma})_{\sigma}$ and $u_{\sigma} = N(I)^{-1/n}$ for all $\sigma$.
\end{defi}

\begin{defi}
 A fractional ideal $I$ is called \textit{reduced} if $1 $ is minimal in $I$.\\
  An Arakelov divisor $D$ is called \textit{reduced} if $D$ has the form $D = d(I)$  for some reduced fractional ideal $I$.
  \end{defi}

%% file: basic2a.tex
\subsection{The Arakelov class group}\qquad\\
The set of all Arakelov divisors of degree 0 form a group, denoted by $\Div^0_F$. It contains the subgroup  of principal divisors. Similar to an algebraic  curve,  we have the following definition.
\begin{defi}
The \textit{Arakelov class group} $\Pic^0_F$ of $F$ is the quotient of $\Div^0_F$ by its subgroup  of principal divisors.
\end{defi}

Denote by 
$(\oplus_{\sigma}\mathbb{R})^0 = \{ (x_{\sigma}) \in \oplus_{\sigma}\mathbb{R}: \sum_{\sigma }deg(\sigma)x_{\sigma} =0 \}$ and $\Lambda = \{(log|\sigma(\varepsilon)|)_{\sigma}: \varepsilon \in O_F^*\}$. Then $\Lambda$ is contained in  the vector space $(\oplus_{\sigma}\mathbb{R})^0$. Let
$$\To = (\oplus_{\sigma}\mathbb{R})^0 / \Lambda.$$
Then by Dirichlet's unit theorem, $\To$ is a compact real torus of dimension $r_1+r_2-1$ {\cite[Section 4.9]{ref:7}}. 
Each $(x_{\sigma})  \in \To$ is mapped to the class of divisors $(O_F, (e^{-x_{\sigma}}))$ in $\Pic^0_F$. In this way, $\To$ becomes a subgroup of $\Pic^0_F$. 
Denoting by $Cl_F$ the class group of $F$, the structure of $\Pic^0_F$ can be seen by the proposition below.
\begin{prop}\label{prop:structure1}
Mapping a divisor class $(I,u) $ to the class of ideal $I$ induces the following exact sequence.
\begin{align*}
 0 \longrightarrow \To \longrightarrow \Pic^0_F \longrightarrow Cl_F \longrightarrow 0.
\end{align*}
\end{prop}
\begin{proof}
See {\cite[Proposition 2.2]{ref:4}}.
\end{proof}

For $u \in \prod_{\sigma}{\mathbb{R}_{>0}}$, we let $\log{u}$ denote the element 
$\log u: = (\log(u_{\sigma}))_{\sigma} \in \prod_{\sigma}\mathbb{R} \subset F_\mathbb{R }.$
By using the scalar product from $F_\mathbb{R }$, this vector has length
$\|\log u\|^2 = \sum_{\sigma } deg(\sigma)|\log(u_{\sigma})|^2 .$
So, we define
\begin{equation*}
  \|u\|_{Pic} =
  \min_{\substack{u' \in \prod_{\sigma}{\mathbb{R}_{>0}}\\
                  \log u' \equiv  \log u (\text{ mod } \Lambda)}}
        \|\log u'\| = \min_{\varepsilon \in O_F^*}    \|\log(|\varepsilon| u)\|. 
\end{equation*}

Now let $[D]$ and $[D']$ be two classes containing divisor $D$ and $D'$ respectively and lying on the same connected component of $\Pic^0_F$. Then by
By Proposition \ref{prop:structure1}, there is some unique  $ u \in \To$ such that $D - D' = (O_F, u)$. We define the \textit{distance} between 2 divisor classes containing $D$ and $D'$ as $\|u\|_{Pic}$.
 
The function $\|  \hspace{0.3cm}\|_{Pic}$ gives rise to a distance function that induces the natural topology of $\Pic^0_F$ 
{\cite[Section 6]{ref:4}}.

%% file: basic2b.tex
\subsection{The oriented Arakelov class group}

\begin{defi}
An \textit{oriented Arakelov divisor } is a pair $(I,u)$ where $I$ is a fractional ideal and $u$ is an arbitrary unit in $F^*_{\mathbb{R}}$. 
\end{defi}

The degree of an oriented Arakelov divisor $D=(I,u)$ is defined by the degree of the Arakelov divisor $(I, |u|)$.
A \textit{principal} oriented Arakelov divisor has the from $(f)= (f^{-1}O_F, f)$ for some $f \in F^*$ where 
the second part in its Hermitian line bundle is $f := (\sigma(f))_{\sigma} \in F^*_{\mathbb{R}}$. 
The set of oriented Arakelov divisors of degree $0$ form a group denoted by $\Divw^0_F$. It contains the subgroup of principal oriented divisors.

\begin{defi}
The quotient of the group $\Divw^0_F$ by the subgroup of principal oriented divisors is called the \textit{oriented Arakelov class group}. It is denoted by $\Picw^0_F$.
\end{defi}

 Let $F^*_{\mathbb{R},conn}$ denote the connected component of $ 1 \in F^*_{\mathbb{R}}$. Then it is isomorphic to
     $\prod_{\sigma \text{ real } }\mathbb{R}^*_+  \times  \prod_{\sigma \text{ complex } }\mathbb{C}^*.$
 We denote by $ F^*_+ = \{ f \in F^*: \sigma(f)>0 \text{ for all real } \sigma \}$ and $ O^*_{F, +} = \{ \varepsilon \in O_F^*: \sigma(\varepsilon) > 0 \text{ for all real } \sigma \}$, subgroups of $F^*_{\mathbb{R},conn}$ and put 
$$(F^*_{\mathbb{R},conn})^0 = \{ u \in F^*_{\mathbb{R},conn}: N(u)=1 \} \qquad \text{   and   } \qquad \Tow = (F^*_{\mathbb{R},conn})^0/O^*_{F, +}. $$
 The \textit{ narrow ideal class group } is a finite group defined as $Cl_{F, +} = Id_F/\{f O_F: f \in F^*_+\}  .$

The following proposition says that the groups 
 $\Tow$ is the connected components of identity of 
  $\Picw^0_F$. It provides an analogue to Proposition \ref{prop:structure1}.

\begin{prop}\label{prop:structure2}
The natural sequence below is exact.
\begin{align*}
 0 \longrightarrow \Tow \longrightarrow \Picw^0_F \longrightarrow Cl_{F, +} \longrightarrow 0  
\end{align*}

 \end{prop}
\begin{proof}
See {\cite[Proposition 5.3]{ref:4}}.
\end{proof}

By Dirichlet's unit theorem,  $\Tow$ is a compact torus of dimension $n - 1$ {\cite[Section 4.9]{ref:7}}.
For $u \in F_{\mathbb{R}}^*$, let $y$ denote the element 
$y= \log u: = (\log(u_{\sigma}))_{\sigma} \in \prod_{\sigma}F_{\sigma} = F_\mathbb{R }$. Here we use the principal branch of the complex logarithm. Then we define
$$\|u\|^2_{\widetilde{Pic}}: = \min_{\varepsilon \in  O^*_{F,+}} \|log(\varepsilon u)\|^2 = 
\min_{\varepsilon \in  O^*_{F,+}} \sum_{\sigma }deg(\sigma)|\log(u_{\sigma} \sigma(\varepsilon))|^2 .$$

Let two divisor classes $D$ and $D'$ that lie on the same connected component of $\Picw^0_F$.
By Proposition \ref{prop:structure2}, there is some unique  $ u  \in \Tow$ such that $D - D' = (O_F, u)$. We define the \textit{distance} between two classes of $D$ and $D'$ as $\|u\|_{\widetilde{Pic}}$. The function $\|  \hspace{0.3cm}\|_{\widetilde{Pic}}$ gives rise to a distance function that induces the natural topology of $\Picw^0_F$. So, $\Picw^0_F0$ is in this way equipped with a translation invariant Riemannian structure {\cite[Section 6]{ref:4}}.

%% file: basic3.tex
The purpose of this section is to introduce strongly $C$-reduced Arakelov divisors of a number field $F$ and to demonstrate some examples. See Chapter 3 in \cite{ref:36} for more details.
\begin{defi}
Let $I$ be a fractional ideal. Then 1 is called  
\textit{primitive } in $I$ if $1 \in I$ and it is not divisible by any integer $d \geq 2$.
\end{defi}

\begin{defi}
	Let $C \geq 1$.  A fractional ideal $I$ is called strongly $C$-\textit{reduced} if the following hold:
\begin{itemize}
\item $1 \in I$  is primitive and
\item  $\text{ For all } g \in I\backslash \{0\}$, we have $ \|1\| \leq C\|g\| .$
\end{itemize}

 The second condition  can be restated as follows: the shortest vector of the lattice $I$ has length at least $ \frac{\sqrt{n}}{C}$. In other words, let $S$ be the sphere centered at the origin and radius $\frac{\sqrt{n}}{C}$. Then the interior of $S$ does not contain any nonzero points of the lattice $I$. Observe that if $I$ is strongly $C$-reduced then $I$ is  strongly  $C'$-reduced for any $C' \geq C$.
\end{defi}

\begin{ex}
Let $F = \mathbb{Q}(\sqrt{7})$. Then $n=2$ and the ring of integers of $F$ is $O_F = \mathbb{Z}[\sqrt{7}]$.
Let $I = O_F + \alpha O_F$ with $\alpha =\frac{1+\sqrt{7}}{4}$. Then $1$ is primitive in $I$. The vector $\alpha$ is a shortest vector of the lattice $I$ with length $\| \alpha\| = 1$. If $C=1$, then the interior of the circle centered at the origin and radius $\sqrt{n}/C = \sqrt{2}$  contains $\alpha$, hence $I$ is not strongly $C$-reduced. In case $C=2$, the interior of the circle centered at the origin and radius $\sqrt{n}/C = \frac{\sqrt{2}}{2}$  does not contain any nonzero points of the lattice $I$, so $I$ is strongly $C$-reduced. (See Figure \ref{pic:1} and Figure  \ref{pic:2}).

\end{ex}
 
 \begin{figure}[h]
 \centering
        \begin{minipage}{0.49\textwidth}
          \includegraphics[width=\linewidth]{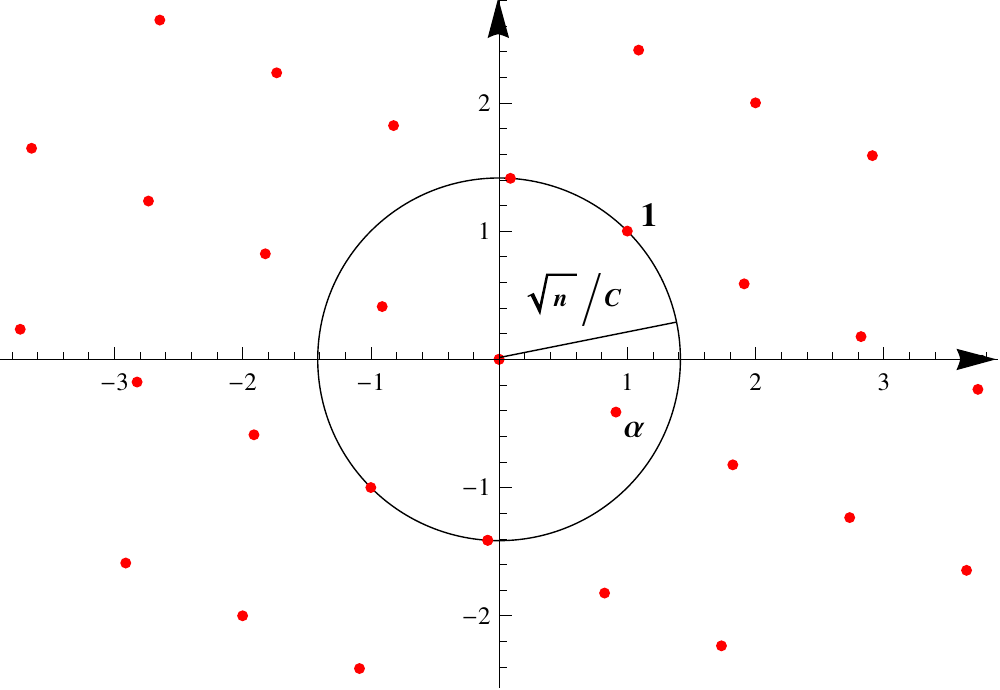}
          \caption{$C=1$: $I$ is not strongly $C$-reduced.}
          \label{pic:1}
        \end{minipage}
        \begin {minipage}{0.49\textwidth}
          \includegraphics[width=\linewidth]{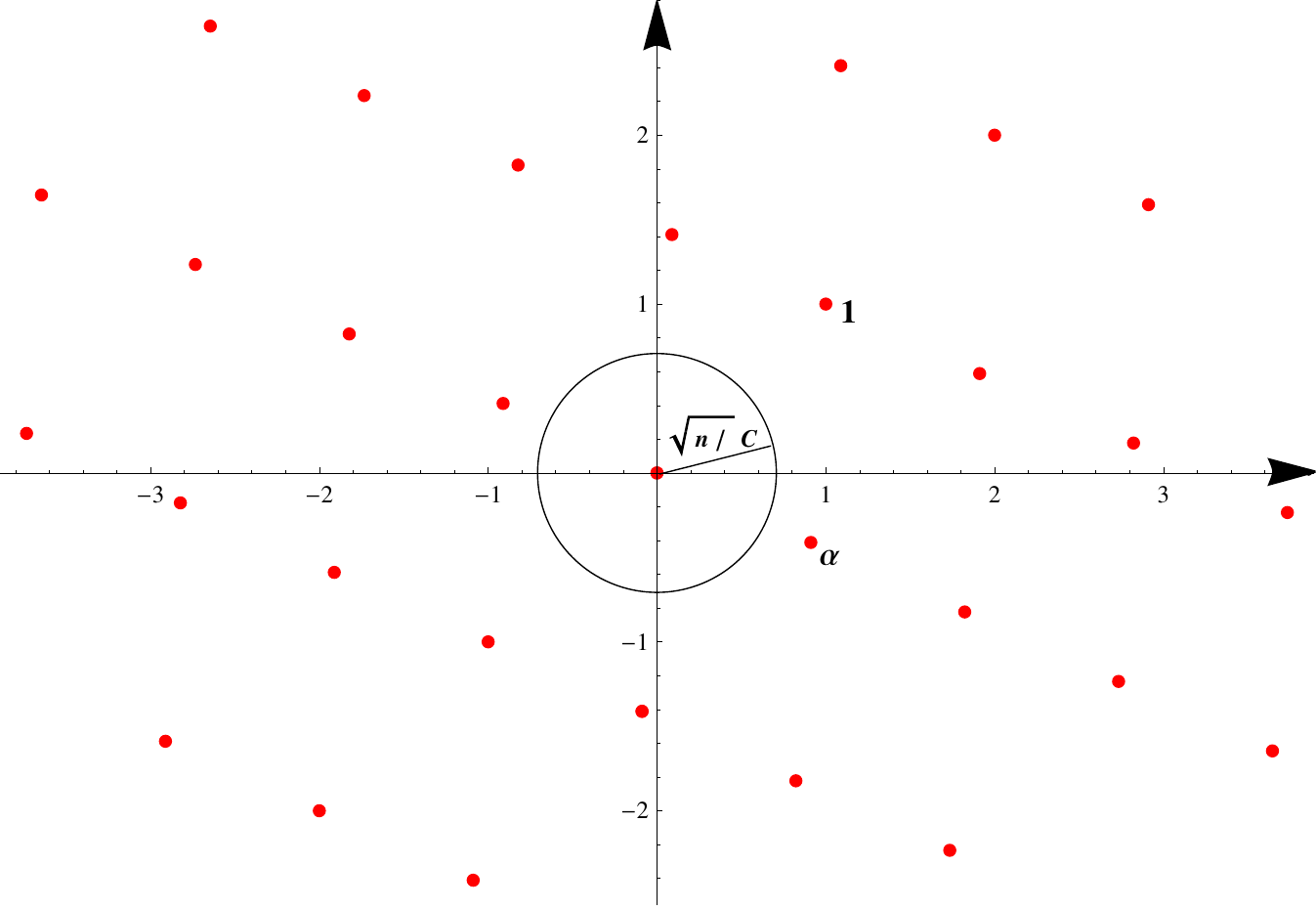}
         \caption{$C=2$: $I$ is strongly  $C$-reduced.}
          \label{pic:2}
        \end{minipage}
     \end{figure}

\begin{defi}
An Arakelov divisor $D$ is called \textit{strongly} $C$-\textit{reduced} if $D$ has the form $D = d(I)$  for some strongly $C$-reduced fractional ideal $I$.
\end{defi}

\begin{ex}\label{excred} Let $F$ be a number field of degree $n$ and let $I$ be a fractional ideal of $F$.
\qquad
\begin{itemize}
\item If 1 is the shortest vector of the lattice $I$, then $d(I)$ is strongly $C$-reduced with $C =1$. In particular, $D = (O_F, 1)$ is strongly 1-reduced.
\item If $I$ is reduced in the usual sense, in other words if 1 is minimal in $I$ (see Section \ref{sec21} or see Section 7 in \cite{ref:4}), then $d(I)$ is strongly $C$-reduced with $C=\sqrt{n}$.
\item Let $F=\mathbb{Q}(\sqrt{\Delta})$ be a real quadratic field  with $\Delta = 73$. There are totally 11 strongly $C$-reduced Arakelov divisors on the principal circle $\To$ of $\Pic^0_F$ with $C= \sqrt{2}$. They are symmetric through the vertical line passing $D_0$ (see Figure \ref{fig:Creqd73}). Moreover, nine divisors are reduced in the usual sense and two divisors $D_3, D_8$ are strongly $C$-reduced with $C=\sqrt{2}$ but not reduced in the usual sense (see \cite{ref:4}).			
\item Let $b_1, ..., b_n$ be an LLL-reduced basis of $I$ and let $J = b_1^{-1}I$. Assume that  1 is primitive in $J$. Then $d(J)$ is strongly $C$-reduced with $C=2^{(n-1)/2} \sqrt{n}$ (see \cite{ref:1} and {\cite[Chapter 3]{ref:36}}).
\begin{figure}
		\begin{center}
		\includegraphics[width=6cm]{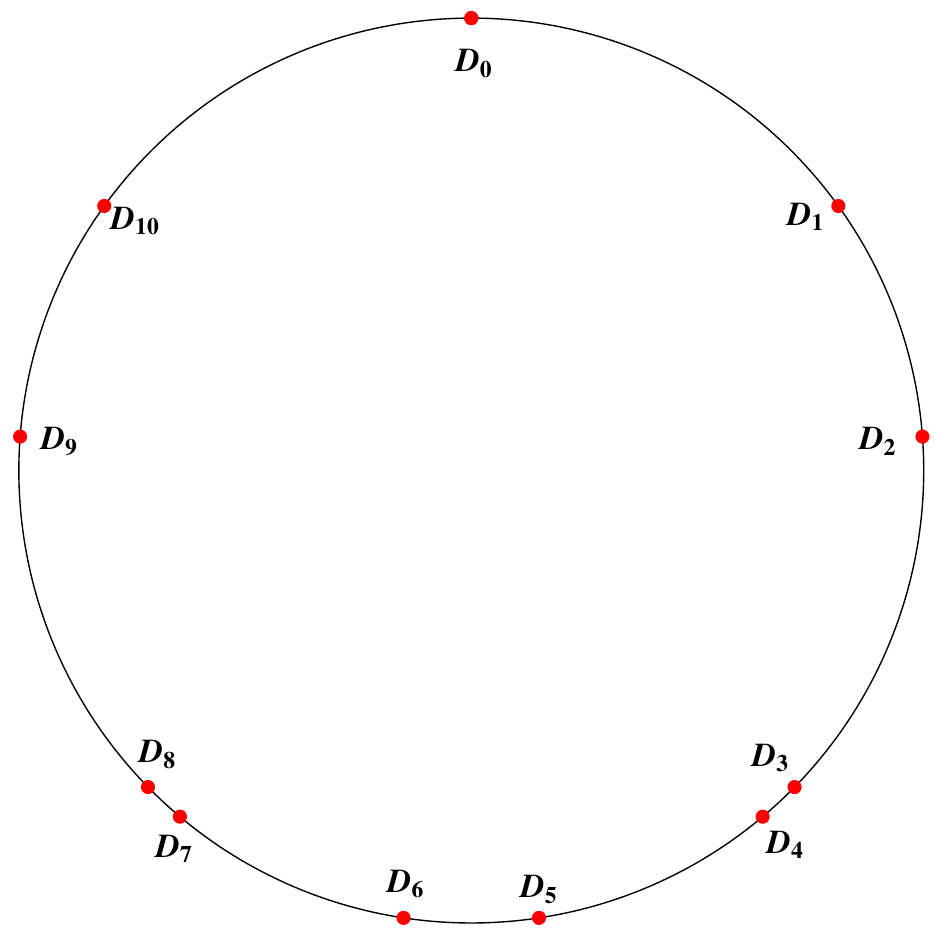}
		\caption{Strongly $C$-reduced Arakelov divisors of $\mathbb{Q}(\sqrt{73})$  with $C =\sqrt{2}$.	\label{fig:Creqd73}}
		\end{center}
\end{figure}

 \end{itemize}
\end{ex}

%% file: pro1.tex
In this section,  let $F$ be a number field of degree $n$ admitting $r_1$ real infinite primes and $r_2$ complex infinite primes. Denote by $\partial_F = \left(\frac{2}{\pi}\right)^{r_2}\sqrt{|\Delta|}$.
We first claim that the set $Sred_F^{C}$ of all strongly $C$-reduced divisors are regularly distributed in the topological groups $\Pic^0_F$ as well as in $\Picw^0_F$. Then we bound for its cardinality.

We first recall the lemma below.

\begin{lem}\label{lem:min}
	Let 
	$D= (I,u)$ be a divisor of degree $0$. Then there is a nonzero element $f \in I$ such that 
	$ u_{\sigma} |\sigma(f)| \leq \partial_F^{1/n} \text{ for all } \sigma$. 
	In particular, we obtain that $\|f\|_{D} \leq \sqrt{n} \partial_F^{1/n}$.
\end{lem}
\begin{proof}
	See {\cite[Proposition 4.4]{ref:4}}.
\end{proof}

\begin{lem}\label{minimal}
	Let 
	$D= (I,u)$ be a divisor of degree $0$. Then there is a minimal element $f \in I$ such that 
	$ u_{\sigma} |\sigma(f)| \leq \partial_F^{1/n} \text{ for all } \sigma$.
\end{lem}

\begin{proof}
	Since $D = (I, u)$ is an Arakelov divisors of degree $0$, by Lemma  \ref{lem:min}, there is a nonzero element $g$ in $I$ such that 
	\begin{equation}\label{eq:min1}
	u_{\sigma} |\sigma(g)| \leq \partial_F^{1/n} \text{ for all } \sigma.
	\end{equation}
	
	We claim that there is a minimal element $f$ of $I$ satisfying this condition. Indeed, if $g$ is minimal then we are done. If $g$ is not minimal then the box 
	$S_g = \{f \in I: |\sigma(f)| < |\sigma(g)| \text{ for all } \sigma \}$ 
	contains some nonzero element $f_1$ of $I$. Hence $ u_{\sigma} |\sigma(f_1)| \leq \partial_F^{1/n} \text{ for all } \sigma$. In other words, we can replace $g$ by $f_1$ in \eqref{eq:min1}. If $f_1$ is minimal then we are done. If not then there is some nonzero element $f_2 $ in $S_g$ satisfying \eqref{eq:min1} and so on. Since the box $S_g$ is bounded, it only contains finitely many elements of the lattice $I$. Therefore, after finite steps, we can find a minimal element $f$ of $I$ satisfying $ u_{\sigma} |\sigma(f)| \leq \partial_F^{1/n} \text{ for all } \sigma$. This proves the claim.
\end{proof}

The following proposition  says that the set of all strongly $C$-reduced Arakelov divisors of a number field is finite. It is similar to Proposition 7.2 in \cite{ref:4}.

\begin{prop}\label{prop:finite}
Let $I$ be a fractional ideal. If $d(I)$ is a strongly $C$-reduced Arakelov divisor, then the inverse $I^{-1}$ of $I$ is an integral ideal  and its norm is at most $C^n \partial_F$.
In particular, the set $Sred_F^{C}$ of all strongly $C$-reduced Arakelov divisors is finite.
\end{prop}

\begin{proof}\qquad
Since $1 \in I$, it follows that $I^{-1} \subset O_F$. 
By Lemma \ref{lem:min}, there is a nonzero element $f \in I$ such that
$$\| N(I)^{-1/n} f\| \leq \sqrt{n} \partial_F^{1/n}.$$
The divisor $d(I)$ is strongly $C$-reduced, thus  $\|f\| \geq \frac{\sqrt{n}}{C}$. Consequently, the following is derived.
$$N(I^{-1})  \leq \frac{\sqrt{n}^n \partial_F }{\|  f\|^n} \leq  \frac{\sqrt{n}^n \partial_F }{\left(\frac{\sqrt{n}}{C}\right)^n}  = C^n \partial_F .$$
Since the number integral ideals of bounded norm is finite, so is the set of all $C$-reduced Arakelov divisors $Sred_F^C$.
\end{proof}

%% file: thm1.tex
We first recall the lemma below {\cite[Lemma 7.5]{ref:4}}.

\begin{lem}\label{sum0}
	Let $x_i \in \mathbb{R}$ for $i=1, 2, ..., n$. Suppose that $\sum_{i=1}^{n}x_i =0$ and that $x \in \mathbb{R}$ has the property that $x_i \leq x$ for all $i =1, 2, ..., n$ then $\sum_{i=1}^{n}x_i^2 \leq n(n-1)x^2$.
\end{lem}

Since this lemma can be proved easily by induction on $n$, we skip the proof here.

The case $C > \sqrt{n}$ is presented by the following theorem. Its result and proof are the same as Theorem 7.4 in \cite{ref:4}. The reason is that if a fractional ideal  is reduced in the usual sense, then it is strongly $C$-reduced with $C \geq  \sqrt{n}$. 

\begin{thm}\label{theo:dis11} 
Let $C\geq\sqrt{n}$. Then for any Arakelov divisor $D = (I,u)$ of degree $0$,  there is a strongly $C$-reduced divisor $D'$ lying on the same connected component of $\Pic^0_F$ as $D$ for which  
  $\|D-D'\|_{Pic}  <  \log{\partial_F}.$
\end{thm}
\begin{proof} See the proof of Theorem 7.4 in \cite{ref:4}.
\end{proof}

%% file: theo12_2.tex
\begin{thm}\label{theo:dis12} Let $1< C \leq \sqrt{n}$. Then for any Arakelov divisor $D = (I,u)$ of degree $0$,  there is a strongly $C$-reduced divisor $D'$ lying on the same connected component of $\Pic^0_F$ as $D$ and  
	$$\|D-D'\|_{Pic}  <  \frac{\log n}{2 \log C} \log{\partial_F}.$$
	
\end{thm}

\begin{proof}
By Lemma  \ref{minimal}, there is a minimal element $f$ in $I$ such that 
	\begin{equation}\label{eq:min}
	u_{\sigma} |\sigma(f)| \leq \partial_F^{1/n} \text{ for all } \sigma.
	\end{equation}
	 
 Let  $J_1 = f^{-1} I$ and $D_1= d(J_1)$. 
 Then $1$ is minimal in $J_1$ and $J_1^{-1}$ is an integral ideal with $1 \leq N(J_1^{-1}) \leq \partial_F$. See Section 7 in \cite{ref:4}. 	
 It is divided into two cases as below.\\
 
\item{\textbf{Case 1:}} $\|1\| \leq C \|f_1\|$ where $f_1$ is the shortest vector of $J_1$. Then $D_1$ is strongly $C$-reduced. Note that $1$ is minimal and hence primitive in $J_1$. We choose $D'= D_1$. 
  Then 
  $$D- D'+ (f) = (O_F,v_1) \text{ with } v_1 = u N(J_1)^{1/n} |f|.$$
  The divisor $D' $ is on the same connected component of $\Pic^0_F$ as $D$ and  	 
  $v_{1\sigma}  = u_{\sigma} |\sigma(f)| N(J_1)^{1/n} $ for all $\sigma$. 
  Hence, 
  $$\log{v_{1\sigma}} = \log{u_{\sigma} |\sigma(f)|} +\log{ N(J_1)^{1/n}} \leq \log{u_{\sigma} |\sigma(f)|}  \text{ for all } \sigma.$$
  The second inequality holds because $N(J_1) \leq 1$. 
  By choosing $f$, each $\log{v_{1\sigma}}$ is bounded by $\log{ \partial_F^{1/n}}$ for all $\sigma$. In addition, 
  since  $\sum_{\sigma}\log{v_{1\sigma}} =0 $, Lemma \ref{sum0} indicates that 
  $$  \| \log{v_1}\|^2 \leq n(n-1)\left(\log{\partial_F^{1/n}}\right)^2 $$
  and so
  $$\|D-D'\|_{Pic} = \|v_1\|_{Pic} \leq \|v_1\|_{Pic} < \log{\partial_F} \leq  \frac{\log n}{2 \log C} \log{\partial_F}.$$

 \item{\textbf{Case 2:}} $\|1\| > C \|f_1\|$. Assume that $k \geq 1$ is the largest integer for which  the shortest vector $f_{j}$ of the lattice $J_{j}$ has the property that  $\|1\| > C \|f_{j}\|$ for all $1 \leq j \leq k$ where  $J_1 = f^{-1} I$ and $J_{i} = f_{i-1}^{-1} J_{i-1}$ if $2 \leq i \leq k$. We claim that $k$ is bounded. 

  Indeed, let $J_{k+1} =  f_k^{-1} J_k$. The arithmetic  geometric mean inequality says that $|N(f_{j})| \leq n^{-n/2} \|f_{j}\|^n < \frac{1}{C^n}$  for $1 \leq j \leq k$ {\cite[Proposition 3.1]{ref:4}}.
  $$N(J_{k+1}^{-1})  = N(f_{k}  \cdots f_1 J_1^{-1}) = |N(f_k)| \cdots |N(f_1)| N(J_1^{-1}) < \frac{1}{C^{n k} } N(J_1^{-1}) \leq \frac{\partial_F}{C^{n k}}.$$
  Since $J_{k+1}^{-1}$ is an integral ideal, we obtain that $N(J_{k+1}^{-1}) \geq 1$. This and the previous inequality imply that $C^{n k} <  \partial_F$. Thus, $ k < \frac{\log \partial_F}{n \log C}$. 
  
  Because of the property of $k$, the divisor $D_{k+1} = d(I_{k+1})$ is strongly $C$-reduced. Moreover, let $D' = D_{k+1}$. Then $D'$ is on the same connected component of $\Pic^0_F$ as $D$ and   
  $$D-D'+ (f_k .... f_2 f_1 f) = (O_F, v_{k+1}).$$
  Here $v_{k+1} = u N(J_{k+1})^{1/n} |f_k \cdots f_1 f| $ and $\|f_j\| \leq \frac{\sqrt{n}}{C} $ for all $j=1, ..., k$. Then using an argument similar to the first case, we obtain that
  $$v_{k+1, \sigma}  \leq \|f_k\| \cdots  \|f_1\| u_{\sigma} |\sigma(f)| N(J_{k+1})^{1/n} \leq  \left(\frac{\sqrt{n}}{C}\right)^k \partial_F^{1/n} \text{  for all } \sigma, $$
  and so
  $$ \|D-D'\|_{Pic}=\|v_{k+1}\|_{Pic} < k n \log{\frac{\sqrt{n}}{C}} + \log{\partial_F} < \frac{\log n}{2 \log C} \log{\partial_F}.$$
  The second inequality comes from the fact that  $ k < \frac{\log \partial_F}{n \log C}$.  Thus, the proposition is proved.		
		
\end{proof}

\begin{rem}
Note that Theorem \ref{theo:dis11} and \ref{theo:dis12} agree on the case $C=\sqrt{n}$. \\
	In case $C=1$, using the same notations and a similar argument as in the proof of Theorem \ref{theo:dis12} leads to the following result. 
		$$\|D-D'\|_{Pic}  <  \text{( some constant) } \cdot \sqrt{|\Delta_F|}.$$ 
Currently, determining whether the distance from  $D'$ to $D$ is bounded above by a polynomial in $\log{\partial_F}$ is still an open problem. 
\end{rem}	

%% file: thm2.tex
The second theorem is a new version of Theorem 7.7 in \cite{ref:4} for strongly $C$-reduced divisors. It says that the set of strongly $C$-reduced divisors is quite sparse in $\Picw^0_F$. Explicitly, the distance between two distinct strongly $C$-reduced Arakelov divisors lying on the same component of $\Picw^0_F$ is at least $\log{\left(1+\frac{\sqrt{3}}{2C^2} \right)}$. This distance depends on $C$ but not on the number field.

\begin{thm}\label{theo:dis2} 
Let $C \geq 1$ and let $\delta = \log{\left(1+\frac{\sqrt{3}}{2C^2} \right)}$. Then we have the following.
\begin{itemize}
\item[(i)] Let $D$ and $D'$ be two strongly $C$-reduced Arakelov divisors in $\widetilde{Div}_F^0$. If there exists an element
$f \in F^*_+$ for which
$D - D' +(f) = (O_F, v)$ 
with $|\log{v_{\sigma}}| < \delta  \text{ for each } \sigma$
then $D = D'$ in $\widetilde{Div}_F^0$. \\
In particular, if $\|v\|_{\widetilde{Pic}} < \delta $, then $D = D'$ in $\widetilde{Div}_F^0$.

\item[(ii)] The natural map from\\
$\bigcup_{D \in Sred_F^C} \left\{ D +(O_F,v): v \in (F^*_{\mathbb{R},conn})^0 \text{ and } |\log{v_{\sigma}}| < \frac{1}{2}\delta  \text{ for each } \sigma \right\}$\\
to $\Picw^0_F$ is injective. 
\end{itemize}  
 
\end{thm}

\begin{proof}\qquad
\item[(i)] 
Suppose that $D= d(I)$ and $D' = d(I')$ are two strongly $C$-reduced divisors with $D - D' +(f) = (O_F, v)$
for some $f \in F^*$ such that $\sigma(f) >0$ for all real $\sigma$.
So, the images of $D$ and $D'$ lie on the same connected component of $\Picw^0_F$. It also implies that  
$\|D-D'\|_{\widetilde{Pic}} = \|v\|_{\widetilde{Pic}} \text{ and } I = f I'  .$

Let  $\lambda = N(I/I')^{\frac{1}{n}} = |N(f)|^{\frac{1}{n}}$.
By changing $I$ and $I'$ if necessary, we can assume that $\lambda \leq 1$. Then $v = N(I/I')^{-\frac{1}{n}} f = \frac{1}{\lambda} f$. Therefore, $\frac{\sigma(f)}{\lambda} = v_{\sigma}$  and $|\log{v_{\sigma}}| < \delta$ for all $\sigma$. 
In consequence, we have the following.
$$|\sigma(f)/\lambda-1| = |v_{\sigma}-1|= |\exp{(\log{v_{\sigma}})}-1| = \exp{|\log{(v_{\sigma})}|-1} < e^{\delta} -1.$$
The assumption  $\lambda \leq 1$ leads to 
$|\sigma(f)-\lambda|< (e^{\delta} -1) \lambda \leq e^{\delta} -1 \text{ for every  } \sigma.$ 
Thus $\|f-\lambda \cdot 1\| < \sqrt{n}(e^{\delta} -1)$.

Suppose that  $1$ and $f$ are $\mathbb{R}$- linearly independent then 
$L= \mathbb{Z} \cdot 1 + \mathbb{Z} \cdot f \subset F_{\mathbb{R}} $ 
 is a lattice of rank $2$. 
And let $d(f,\mathbb{R} \cdot 1)$ denote the distance from $f$ to $\mathbb{R} \cdot 1$, the 1-dimensional subspace of $F_{\mathbb{R}}$ spanned by $1$. Since  
$$ d(f,\mathbb{R} \cdot 1) \leq \|f-\lambda \cdot 1\|< \sqrt{n}(e^{\delta} -1),$$
the lattice $L$ has covolume 
$$\co(L) = \|1\| \cdot d(f,\mathbb{R} \cdot 1) < n (e^{\delta} -1). $$
On the other hand,  $L$ contains a nonzero element $g$ such that
$$\|g\|^2 \leq \left(\frac{4}{3}\right)^{1/2} \co(L) $$
{\cite[Section 9]{ref:1}}.
As a result, we obtain that
 $$\|g\|^2< \frac{2}{\sqrt{3}}  n (e^{\delta} -1).$$
Both $1 $ and $f$ are in $I$, so then $g$. By assumption, $D = d(I)$ is strongly $C$-reduced, hence $\|g\| \geq \sqrt{n}/C$ and
$$\frac{n}{C^2} \leq \|g\|^2  < \frac{2}{\sqrt{3}}  n (e^{\delta} -1) = \frac{n}{C^2}.$$
This contradiction follows from the assumption that  $1$ and $f$ are $\mathbb{R}$-linearly independent.
Therefore, $1$ and $f$ are $\mathbb{R}$-linearly dependent. Then there is some $r \in \mathbb{R}$ for which  $\sigma(f) = r \cdot 1 = r$ for all $\sigma$. This implies that $f \in \mathbb{Q}$. Since $1$ is primitive in $I$ and $f \in I$, we must have $f \in \mathbb{Z}$. 
Accordingly,  $|f| = |N(f)|^{\frac{1}{n}} = \lambda \leq 1$.  As a result,  $f= \pm 1$ and $I = I'$. It follows that $ D = D'$ in $\widetilde{Div}_F^0$. That completes the proof of the first part in (i).\\

In particular, if $\|v\|_{\widetilde{Pic}} < \delta $, then $\min_{\varepsilon \in O^*_{F,+}}\| \log(v \varepsilon)\| < \delta $. Consequently, there is some element $ \varepsilon \in O_F^*$ for which $\sigma(\varepsilon)>0$ for all real $\sigma$ and 
$| \log(v_{\sigma} \sigma(\varepsilon))| < \delta $ for all $\sigma$. Replacing $f$ by $\varepsilon f  \in F^*_+$, the following is obtained.
$$D - D' =(\varepsilon f) + (O_F, v \varepsilon)  \text{ and } | \log(v_{\sigma} \sigma(\varepsilon))| < \delta   \text{ for each } \sigma.$$
Then the result follows by applying the first statement in (i), which is proved above.

\item[(ii)] Assume that two divisors  $ D_1 = D +(O_F,v_1)$ and $D_2 = D' +(O_F,v_2)$ are in the left set in (ii) and  have the same image in $\Picw^0_F$. Then $D$ and $D'$ are strongly  $C$-reduced divisors,  $v_1 =(v_{1\sigma})_{\sigma}$ and $v_2 =(v_{2\sigma})_{\sigma}$ are in $ (F^*_{\mathbb{R},conn})^0 $ satisfying  $ |\log{v_{1\sigma}}| < \frac{1}{2} \delta$  and $ |\log{v_{2\sigma}}| < \frac{1}{2} \delta$ for all $\sigma$.  Moreover, 
$D_1 - D_2 = (f^{-1} O_F, f  ) =(f) \text{ for some } f \in F^*_+ .$\\
Thus, $D-D'= (f ) + (O_F,v)  $ where $f \in F^*_+ $ and $v = v_1 v_2^{-1}$. In addition, $|\log v_{\sigma}|= |\log v_{1\sigma} - \log v_{2\sigma}| < \frac{1}{2} \delta + \frac{1}{2} \delta = \delta$  for all $\sigma$. 
By part $(i)$, two divisors $D $ and $D'$ must coincide in $\widetilde{Div}_F^0$. As a sequence, 
$(f^{-1} O_F, f  )  = (O_F, v) $ in $\widetilde{Div}_F^0$. Therefore, $ f \in O_F^*  \text{ and so } D_1 = D_2 $.
This claims that the map in $(ii)$ is injective.

\end{proof}

%% file: cardired.tex
Theorem \ref{theo:dis2} has the following corollary.

\begin{cor}
Let $C \geq 1$. Then the number of  strongly $C$-reduced Arakelov divisors contained in a ball of radius $1$ in $\Pic^0_F$ is at most $\left(\frac{1}{2}\log{\left( 1+\frac{3}{2C^2}\right)} \right)^{-n} $.
\end{cor}
\begin{proof}
We denote $ \log{\left( 1+\frac{3}{2C^2}\right)}$ briefly by $\delta$. Let $B$ be a ball of radius $1$ in $\Pic^0_F$. And let $B_R = B \cap Sred_F^C$ be the set of  strongly $C$-reduced divisors whose images are contained in $B$. Then $B_R$ is contained in a subset $S$ of $\Picw^0_F$ of volume  $\frac{2^{r_1}(2 \pi \sqrt{2})^{r_2}}{\omega_F}$ times the volume of a unit ball in $\Pic^0_F$. \\
For each  strongly $C$-reduced divisor $D$ in $B_R$, we denote by $B_{1/2 \delta}(D)$ the ball of radius $\frac{1}{2} \delta$ centered at $D$ in $\Picw^0_F$. By Theorem \ref{theo:dis2}, these balls are mutually disjoint in $S$. Therefore, the followings holds. 
$$\sum_{D \in B_R} \vo(B_{1/2 \delta}(D))\leq \vo(S) .$$
Each ball $B_{1/2 \delta}(D)$ has volume equal to the volume of the ball $B_{1/2 \delta}$ centered at the origin and radius $1/2 \delta$ in $\Picw^0_F$. Hence
$$\#B_R  \cdot  \vo(B_{1/2 \delta})\leq \vo(S) .$$

Besides, the balls $B_{1/2 \delta}$ have covolume $\frac{2^{r_1}(2 \pi \sqrt{2})^{r_2}}{\omega_F} \left(\frac{1}{2} \delta\right)^n$ times the volume of a unit ball in $\Pic^0_F$.
This leads to the following.
$$\#B_R \leq \frac{\vo(S)}{\vo(B_{1/2 \delta})}  = \left(\frac{1}{\frac{1}{2} \delta} \right)^n = \left(\frac{1}{2} \delta  \right)^{-n} .$$
So, the corollary is proved.
\end{proof}

Finally, we bound for $\#Sred_F^C$ as below.
\begin{cor}
Let $C \geq 1$. Then  
$$\#Sred_F^C \leq 2^n \left(\log{\left( 1+\frac{3}{2C^2}\right)}\right)^{-n/2}  \vo(\Picw^0_F).$$
\end{cor}
\begin{proof}
Let $S$ be the simplex given by 
$$S = \{ v \in \left(F_{\mathbb{R}, conn}\right)^0: |\log{v_{\sigma}}| < \frac{1}{2}\delta \text{ for all } \sigma\}.$$
Then its volume is
$$\vo(S) = \frac{2^{-r_2/2} n^{r_1+r_2-1/2}}{(r_1+r_2-1)!}\left(\frac{1}{2}\delta \right)^{r_1+r_2} \geq 2^{-n}\delta^{n/2} .$$
Furthermore, the second  part of Theorem \ref{theo:dis2} implies that 
$$\vo(\Picw^0_F) \geq \#Sred_F^C  \cdot \vo(S).$$
Therefore, the result follows.
\end{proof}

%% file: main.bbl
\begin{thebibliography}{10}

\bibitem{ref:0}
Eva Bayer-Fluckiger.
\newblock Lattices and number fields.
\newblock In {\em Algebraic geometry: {H}irzebruch 70 ({W}arsaw, 1998)}, volume
  241 of {\em Contemp. Math.}, pages 69--84. Amer. Math. Soc., Providence, RI,
  1999.

\bibitem{ref:9}
Johannes Buchmann.
\newblock A subexponential algorithm for the determination of class groups and
  regulators of algebraic number fields.
\newblock In {\em S\'eminaire de {T}h\'eorie des {N}ombres, {P}aris
  1988--1989}, volume~91 of {\em Progr. Math.}, pages 27--41. Birkh\"auser
  Boston, Boston, MA, 1990.

\bibitem{ref:10}
Johannes Buchmann and H.~C. Williams.
\newblock On the infrastructure of the principal ideal class of an algebraic
  number field of unit rank one.
\newblock {\em Math. Comp.}, 50(182):569--579, 1988.

\bibitem{ref:5}
H.~W. Lenstra, Jr.
\newblock On the calculation of regulators and class numbers of quadratic
  fields.
\newblock In {\em Number theory days, 1980 ({E}xeter, 1980)}, volume~56 of {\em
  London Math. Soc. Lecture Note Ser.}, pages 123--150. Cambridge Univ. Press,
  Cambridge, 1982.

\bibitem{ref:1}
Hendrik~W. Lenstra, Jr.
\newblock Lattices.
\newblock In {\em Algorithmic number theory: lattices, number fields, curves
  and cryptography}, volume~44 of {\em Math. Sci. Res. Inst. Publ.}, pages
  127--181. Cambridge Univ. Press, Cambridge, 2008.

\bibitem{ref:8}
R.~J. Schoof.
\newblock Quadratic fields and factorization.
\newblock In {\em Computational methods in number theory, {P}art {II}}, volume
  155 of {\em Math. Centre Tracts}, pages 235--286. Math. Centrum, Amsterdam,
  1982.

\bibitem{ref:4}
Ren{\'e} Schoof.
\newblock Computing {A}rakelov class groups.
\newblock In {\em Algorithmic number theory: lattices, number fields, curves
  and cryptography}, volume~44 of {\em Math. Sci. Res. Inst. Publ.}, pages
  447--495. Cambridge Univ. Press, Cambridge, 2008.

\bibitem{ref:7}
Daniel Shanks.
\newblock The infrastructure of a real quadratic field and its applications.
\newblock In {\em Proceedings of the {N}umber {T}heory {C}onference ({U}niv.
  {C}olorado, {B}oulder, {C}olo., 1972)}, pages 217--224. Univ. Colorado,
  Boulder, Colo., 1972.

\bibitem{ref:36}
Ha~T.~N. Tran.
\newblock On reduced arakelov divisors of a number field.
\newblock Doctoral Thesis, University of Rome Tor Vergata, 2015.

\bibitem{ref:33}
Ha~T.~N. Tran.
\newblock On reduced arakelov divisors of real quadratic fields, to appear in
  {\em Acta Arith.}.

\bibitem{ref:18}
Hugh~C. Williams.
\newblock Continued fractions and number-theoretic computations.
\newblock {\em Rocky Mountain J. Math.}, 15(2):621--655, 1985.
\newblock Number theory (Winnipeg, Man., 1983).

\end{thebibliography}
